\documentclass[12pt]{amsart}

\usepackage{amsfonts,amsthm,amsmath,amssymb,amscd,mathrsfs}
\usepackage{graphics}
\usepackage{indentfirst}
\usepackage{cite}
\usepackage{latexsym}
\usepackage[dvips]{epsfig}
\usepackage{color}
\usepackage{bookmark}

\usepackage{lmodern}

\newtheorem{theorem}{Theorem}[section]

\newtheorem{remark}{Remark}[section]
\newtheorem{definition}{Definition}[section]
\newtheorem{lemma}[theorem]{Lemma}

\renewcommand{\div}{{\rm div \thinspace }}

\newcommand{\dis}{\displaystyle}

\newcommand{\bt}{\begin{theorem}}
	\newcommand{\bl}{\begin{lemma}}
		\newcommand{\el}{\end{lemma}}
	\newcommand{\et}{\end{theorem}}

\newcommand{\bn}{\begin{eqnarray}}
	\newcommand{\en}{\end{eqnarray}}
\newcommand{\bnn}{\begin{eqnarray*}}
	\newcommand{\enn}{\end{eqnarray*}}

\newcommand{\ba}{\begin{aligned}}
	\newcommand{\ea}{\end{aligned}}
\newcommand{\be}{\begin{equation}}
	\newcommand{\ee}{\end{equation}}

\newcommand{\Bv}{{\boldsymbol{v}}}

\newcommand{\Bn}{{\boldsymbol{n}}}

\newcommand{\Bu}{{\boldsymbol{u}}}
\newcommand{\Be}{{\boldsymbol{e}}}

\newcommand{\Bf}{{\boldsymbol{f}}}

\newcommand{\Btau}{{\boldsymbol{\tau}}}

\newcommand{\mcL}{\mathcal{L}}
\newcommand{\mfu}{\mathfrak{u}}
\newcommand{\mfv}{\mathfrak{v}}

\newcommand{\Bw}{\boldsymbol{w}}
\begin{document}
	
	\title
	[Rigidity of Steady Solutions to the Navier-Stokes Equations]
	{Rigidity of Steady Solutions to the Navier-Stokes Equations in High Dimensions and its applications}
	
	\author{Jeaheang Bang}
	\address{Department of Mathematics, The University of Texas at San Antonio, One UTSA Circle, San Antonio, Texas, 78249}
	\email{jeaheang.bang@utsa.edu}

	\author{Changfeng Gui}
	\address{Department of Mathematics, Faculty of Science and Technology, University of Macau, Taipa, Macao and   Department of Mathematics, The University of Texas at San Antonio, One UTSA Circle, San Antonio, Texas, 78249}
	\email{Changfeng.Gui@utsa.edu}
	
	\author{Hao Liu}
	\address{School of Mathematical Sciences and Institute of Natural Sciences, Shanghai Jiao Tong University, 800 Dongchuan Road, Shanghai, China}
	\email{mathhao.liu@sjtu.edu.cn}
	
	\author{Yun Wang}
	\address{School of Mathematical Sciences, Center for dynamical systems and differential equations, Soochow University, Suzhou, China}
	\email{ywang3@suda.edu.cn}

	\author{Chunjing Xie}
	\address{School of Mathematical Sciences, Institute of Natural Sciences,
		Ministry of Education Key Laboratory of Scientific and Engineering Computing,
		and CMA-Shanghai, Shanghai Jiao Tong University, 800 Dongchuan Road, Shanghai, China}
	\email{cjxie@sjtu.edu.cn}

	\begin{abstract}
		Solutions with scaling-invariant bounds such as self-similar solutions, play an important role in the understanding of the regularity and asymptotic structures of solutions to the Navier-Stokes equations.
		In this paper, we prove that any steady solution satisfying $|\Bu(x)|\leq C/|x|$  for any constant $C$ in $\mathbb{R}^n\setminus \{0\}$ with $ n \geq 4$, must be zero.
		Our main idea is to analyze the velocity field and the total head pressure via weighted energy estimates with suitable multipliers so that the proof is pretty elementary and short. These results not only give the Liouville-type theorem for steady solutions in higher dimensions with neither smallness nor self-similarity type assumptions, but also help to remove a class of singularities of solutions and give the optimal asymptotic behaviors of solutions at infinity in the exterior domains.
	\end{abstract}

	\keywords{}
	\subjclass[]{}

	\maketitle
	
	\section{Introduction and main results}

	Consider the $n$-dimensional steady Navier-Stokes equations
	\begin{align} \label{NS}
		-\Delta \Bu + (\Bu\cdot \nabla )\Bu + \nabla p=0, \quad \div \Bu=0,
	\end{align}
	where   the unknowns are velocity $\Bu$, an $n$-dimensional vector field and  pressure $p$,  a scalar function.
	Equations \eqref{NS} enjoy special properties: scaling invariance and rotation symmetry. If $(\Bu,p)$ is a solution to \eqref{NS}, then the scaled one
	$ (\lambda \Bu (\lambda x), \lambda^2p(\lambda x)) $
	is also a solution for every $\lambda>0$.
	This important property has proved useful to investigate the Navier-Stokes equations; see \cite{CKN, Lin, TianXin99, SverakTsai00, LiYang22, JiaSverak17} and references therein.
	Motivated by these properties,
	one can study the solutions with scaling-invariant bounds or the solutions in the spaces with scaling-invariant  norms such as $L^n$ and $L^{n,\infty}$ (the weak $L^n$ space) in $n$-dimensional case. Among these solutions, the most special ones are the self-similar solutions, which have been extensively studied  in both steady  and  unsteady cases (cf. \cite{Tsai18}). In this paper, we focus on the steady solutions.
	
	
	In the steady case, a vector field $\Bu$ is said to be  \emph{self-similar } (SS) if and only if for \emph{all} $\lambda>1$, it holds that
	$\Bu(x)=\lambda \Bu(\lambda x)$ in $\mathbb{R}^n \setminus \{0\}$.
	One can relax the requirement for $\lambda$ to consider a more general scaling-invariant class: a vector field $\Bu$ is called \emph{discretely self-similar} (DSS or $\lambda$-DSS) if and only if for \emph{some} $\lambda>1$, it holds that
	$\Bu(x)=\lambda \Bu(\lambda x)$ in $\mathbb{R}^n \setminus \{0\}$.
Furthermore, for $\phi\in \mathbb{R}$ and $\lambda>1$, define
 \begin{equation}\label{eq:RDSS}
      \Bu_{\lambda, \phi}(x)=\lambda \mathcal{R}(-\phi)\Bu(\lambda \mathcal{R}(\phi) x),
 \end{equation}
where, for simplicity, we only consider the rotation matrix around the $x_3$-axis and
 \begin{equation}\label{eq:Rotation}
			\mathcal{R}(\phi)=\left[
			\begin{aligned}
				&\cos \phi  	&-\sin \phi   \quad&0\\
				&\sin \phi  &\cos \phi  \quad&0\\
				&\quad 0      &0\quad\quad   &1\\
			\end{aligned}
			\right].
		\end{equation}
 A vector field $\Bu$ is said to be \emph{rotated self-similar} (RSS) and \emph{rotated discretely self-similar} (RDSS),  if for all $\lambda>1$ and some $\lambda>1$, respectively, $\Bu_{\lambda,\phi}(x)=\Bu(x)$ on $\mathbb{R}^n \setminus \{0\}$ with $\phi=2\alpha \log \lambda$ for some fixed $\alpha\in \mathbb{R}$. 
 For more discussions on these scaling invariant solutions, one can refer to \cite[Section 8.1]{Tsai18}.
	We  note that  there are nontrivial explicit examples of steady self-similar solutions in $\mathbb{R}^n\setminus \{0\}$ when $n=2,3$: Hamel solutions ($n=2$, \cite{GuillodWittwer15, Landau}), and Landau solutions ($n=3$, \cite{Sverak11, Tian98, Landau}).
	And they are the only ones among the self-similar solutions in the two and three-dimensional cases respectively (cf.\cite{Sverak11}).
	It is very interesting that for $n=2$   a class of rotated self-similar solutions which are not self-similar, has been constructed in \cite{GuillodWittwer15SIMA}. On the other hand, it is known that for $n\geq 4$, all self-similar solutions are trivial (cf.  \cite{Tsai98b}, \cite[Corollary in Section 2]{Sverak11}).
	
	Note that all these kinds of scaling-invariant vector fields $\Bu$ satisfy the  bound \begin{align}\label{eq:sacinv}
		|\Bu(x)|\leq \frac{C}{|x|}, ~  x \neq 0
	\end{align}
	for some constant $C>0$ provided that they are bounded on the sphere $|x|=1$ or the annulus $1\leq |x|\leq \lambda$.
	In this paper, we consider the solutions satisfying only the  bound \eqref{eq:sacinv} without  imposing any type of self-similarity. Note that the assumption \eqref{eq:sacinv} is equivalent to $|x|\Bu(x)\in L^\infty (\mathbb{R}^n\setminus \{0\})$.
	We denote this class of functions by $\mcL^\infty_r$. By the defintion, one can check that the following 
	 inclusion property  hold (\hspace{1sp}\cite{Tsai18}),
	\begin{align*}
		\text{SS}
		\subset
		\text{RSS}
		\subset
		\text{DSS}
		\subset
		\text{RDSS}
		\subset
		\mcL^\infty_r.
	\end{align*}
	Classifying the solutions which satisfy \eqref{eq:sacinv} has various implications in studying the singularity and  asymptotics of the solution to the Navier-Stokes system \eqref{NS}.
    Hence, motivated by the above-mentioned rigidity results for the self-similar solutions,
	\u Sver\'ak  raised a conjecture in \cite{Sverak11} for the three-dimensional case: whether any solution $\Bu$ satisfying \eqref{eq:sacinv} in $\mathbb{R}^3\setminus \{0\}$ is a Landau solution.

	The main objective of this paper is to consider the analogous conjecture in higher dimensions $n\geq 4$, i.e., to classify the solutions of the steady Navier-Stokes system \eqref{NS} which satisfy \eqref{eq:sacinv} in higher dimensions. In higher dimensional cases ($n\geq 4$),  no explicit example has been known in the class $\mcL^{\infty}_r$ other than the trivial solutions, and in particular, self-similar ones are trivial as mentioned above.
	Furthermore, it was proved in \cite{JiaSverak17} that the leading terms in the asymptotic expansions at infinity of finite energy solutions are determined by the fundamental solution to the linearized equations, i.e., the Stokes equations, at least when the dimension $n=4,5,6$.
	The first main result of this paper is to show rigorously that all solutions in  $\mcL^{\infty}_r$ are trivial in all  higher dimensional cases ($n\geq 4$);  see Theorem \ref{main}. The second main goal of this paper is to investigate the singularities and the asymptotics at large distances of  solutions to the Navier-Stokes system  based on this Liouville type result; see Theorems \ref{cor:regularity} and  \ref{thm:asy} for the precise statements.
	
	Our following main result shows the rigidity of steady solutions to the Navier-Stokes system \eqref{NS} in the class $\mcL_r^\infty$ in the case dimension $n\geq 4$.
\begin{theorem} \label{main}
Let $n\geq 4$ and $\Bu$ be a smooth solution of the steady Navier-Stokes system \eqref{NS} in $\Omega$ with $\Omega =\mathbb{R}^n\setminus \{0\}$ or $\mathbb{R}^n_+\setminus \{0\}$ where $\mathbb{R}^n_+$ is the upper half-space in $\mathbb{R}^n$. Then $\Bu\equiv 0$ in the following three cases.
\begin{enumerate}
   \item[(i)] $\Omega =\mathbb{R}^n\setminus \{0\}$  and $\Bu \in C^\infty(\Omega)\cap \mcL^{\infty}_r$, i.e.,
	there exists a constant $C>0$ such that $\Bu$ satisfies
	\begin{align} \label{eq4}
			|\Bu(x)|\leq \frac{C}{|x|}
			\quad \text{in }\Omega.      \end{align}

\item[(ii)]  $\Omega= \mathbb{R}^n_+\setminus\{0\}$ and $\Bu$ satisfies the Navier slip boundary condition
		\begin{align} \label{eq4_0}
		\Bu\cdot \Bn = 0, \ \ \ \Bn\cdot D(\Bu)\cdot \Btau = 0	 \quad \text{on }\partial\Omega \setminus \{0\}
		\end{align}
		and
		\eqref{eq4}
		for some constant $C>0$, where $D(\Bu)$ is the  tensor defined by 
  \begin{equation*}
      D(\Bu)_{i, j} = \frac12(\partial_i u_j + \partial_j u_i), 
  \end{equation*}
  $\Bn$ and $\Btau$ are the normal and tangential vector on $\partial \Omega$, respectively.

	\item[(iii)] $\Omega= \mathbb{R}^4_+\setminus\{0\}$ and $\Bu$ satisfies the Dirichlet boundary condition
		\begin{align} \label{eq4_1}
			\Bu=0 \quad \text{on }\partial\Omega \setminus \{0\}
		\end{align}
		and
		\eqref{eq4}
		for some constant $C>0$. 
\end{enumerate}

	\end{theorem}
	
	We have some remarks for this theorem.

\begin{remark}
	In Theorem \ref{main},  we need neither a smallness assumption on $\Bu$ nor  any type of self-similarity such as RSS, DSS, or RDSS.
\end{remark}

	\begin{remark}
		In the case $n\geq 4$,
		consider a solution $\Bu\in C^\infty (\mathbb{R}^n)$ to \eqref{NS} such that
		\begin{align*}
			|\Bu(x)|= O
			\left(
			\frac{1}{|x|^\alpha}
			\right)
			\quad \text{as }|x|\to\infty, \quad \alpha>0.
		\end{align*}
		As it was mentioned in \cite[Equation (X.9.29)]{Galdi11},
		$\Bu$ must be trivial, i.e., $\Bu\equiv 0$ provided that $\alpha>\frac{n-1}{3}$.
		In fact, with the aid of an iteration argument,
		one can show $|\Bu(x)|=O(1/|x|^{n-2})$ as long as $\alpha>1$ (a weaker condition for $n\geq 5$), so that the solution must be trivial by using energy estimates.

  On the other hand, in the borderline case $\alpha=1$, the same iteration argument no longer works.
   However,  one can utilize Theorem \ref{main} to conclude $\Bu \equiv 0$.  It is worth noting that  it is an outstanding open problem to prove the triviality of a solution $\Bu$ with $\nabla \Bu \in L^2 (\mathbb{R}^3)$ (cf. \cite{Galdi11}).
   
	\end{remark}

 \begin{remark}
		In \cite{Sverak11}, \u Sver\'ak raised a problem on the rigidity of steady {self-similar} solutions in $\mathbb{R}_+^n$ ($n\geq 3$) with the Dirichlet boundary condition. In dimension four, an affirmative answer for this problem was given  in \cite{Shi18}.
		Theorem \ref{main} proves the rigidity of four-dimensional steady solutions in half-space in the class $\mcL^{\infty}_r$ without any additional assumptions.
	\end{remark}
	
	\begin{remark}
		Using similar ideas in the proof of Theorem \ref{main}, one can also  show the triviality of the solutions in a cone in the four-dimensional space under no-slip boundary conditions and similar assumptions.
	\end{remark}

	 Theorem \ref{main} not only provides the rigidity of solutions to the steady Navier-Stokes system \eqref{NS}  but also helps to understand the regularity of weak solutions for the steady Navier-Stokes system. 
	 For readers' convenience and self-contained purpose, we first recall the definitions of weak and very weak solutions here.
	For a smooth  domain $\Omega \subset \mathbb{R}^n$, consider the steady Navier-Stokes equations with an external force $\boldsymbol{f}$:
	\begin{align} \label{NS_2}
		\begin{aligned}
			-\Delta \Bu + \Bu \cdot \nabla \Bu + \nabla p = \boldsymbol{f}, \quad \div \Bu =0\quad \text{in }\Omega.
			\\
		\end{aligned}
	\end{align}
	\begin{definition} \label{def_weak}
		For $\boldsymbol{f}\in L^\infty (\Omega)$, the pair $(\Bu,p)$ is called a \emph{weak solution} to \eqref{NS_2} in $\Omega$  if
		\begin{align*}
			\Bu \in W^{1,2}_{loc} (\Omega),
			\quad p \in W^{1,\frac{n}{n-1}}_{loc}(\Omega),
		\end{align*}
		and it satisfies
		\begin{align} \label{eq8}
			\int_{\Omega}
			\nabla \Bu : \nabla \boldsymbol{\varphi} \thinspace dx
			+ \int_\Omega \Bu \cdot \nabla \Bu \cdot \boldsymbol{\varphi}\thinspace dx
			+ \int_\Omega \nabla p \cdot  \boldsymbol{\varphi} \thinspace dx
			= \int_\Omega \boldsymbol{f} \cdot \boldsymbol{\varphi} \thinspace dx
		\end{align}
		for all vector fields $\boldsymbol{\varphi} \in C^\infty _0 (\Omega)$
		and \begin{align} \label{eq:divfree1}
			\int_{\Omega} \Bu\cdot \nabla \varphi \, dx=0
		\end{align}
		for all $\varphi \in C^\infty _0 (\Omega)$, where $\nabla \Bu : \nabla \boldsymbol{\varphi}= \sum_{i,j} \partial_i u_j \partial_i \varphi_j.$
		
		Furthermore, $\Bu\in L^2_{loc}(\Omega)$ is called a \emph{very weak solution} in $\Omega$ if $\Bu$ satisfies
		\begin{align} \label{eq9}
			-\int_\Omega
			(\Bu \cdot \Delta \boldsymbol{\varphi} + \Bu \cdot (\Bu \cdot \nabla )\boldsymbol{\varphi}) \, dx=\int_\Omega \boldsymbol{f} \cdot \boldsymbol{\varphi} \thinspace dx
		\end{align}
		for all  divergence-free vector fields $\boldsymbol{\varphi} \in C^\infty _0 (\Omega)$ and  \eqref{eq:divfree1} for all ${\varphi} \in C^\infty _0 (\Omega)$.
	\end{definition}
	For more explanations  on weak solutions of the steady Navier-Stokes system,  one can  refer to \cite{FrehseRuzicka94, Galdi11, Tsai18, Sverak11}. We focus on the case with zero external force from now on.
	For a solution $\Bu\in C^\infty (\mathbb{R}^n \setminus \{0\})$ to \eqref{NS} (i.e., $\boldsymbol{f}=0$)   belonging to $\mcL^{\infty}_r$,  one can show that $\Bu$ satisfies \eqref{eq9} across the origin if $n\geq 4$.
 Furthermore,  $\nabla \Bu\in L^2_{\mathrm{loc}}(\mathbb{R}^n)$ and $ p\in W^{1,\frac{n}{n-1}}_{\mathrm{loc}}(\mathbb{R}^n)$ if $n\geq 5$. And it satisfies \eqref{eq8} across the origin if  $n\geq 5$.
	  Hence, a solution $\Bu$  in $\mcL^{\infty}_r$    is a very weak solution in $\mathbb{R}^n$ across the origin  if $n\geq 4$ and is a weak solution in $\mathbb{R}^n$ across the origin  if $n\geq 5$.
	
	Proving better regularity of a very weak solution to \eqref{NS} (cf. Definition \ref{def_weak} above), which only belongs  to $L^2_{loc}$ by definition,  is one of the major open problems in the study of the steady Navier-Stokes equations; see \cite{Sverak11, Galdi11, Tsai18, FrehseRuzicka98, FrehseRuzicka96} and references therein.
	This regularity problem has got much attention because the  five-dimensional steady flows are  considered as a counterpart for the unsteady three-dimensional flows, whose regularity is a famous and challenging problem. For more explanations on this issue, one may refer to  \cite{Struwe95, Tsai18, Sverak11} and  references therein.
	
	The standard regularity theory (cf. \cite[Theorem IX.5.1]{Galdi11}, \cite{KimKozono06, Tsai18}) implies that a very weak solution $\Bu$ which also satisfies $\Bu\in L^n_{\mathrm{loc}}, n\geq 3,$ (or $\Bu \in L^{n+\varepsilon}_{\mathrm{loc}}$, for some $\varepsilon>0, n=2$) must be regular. 
	However, a solution in $\mcL^{\infty}_r$ does not meet this additional condition
	in any dimension and  can be considered a borderline case.
	Moreover, it was proved in \cite{KimKozono06,KPR20JMFM} that: for $n\geq 3$, if  $\|\Bu\|_{L^{n,\infty}(\Omega)}$ is sufficiently small, then $\Bu$ is regular in $\Omega$. A solution in $\mcL^{\infty}_r$ belongs to $L^{n,\infty}$, but it may not be small.
	Hence one can study the solutions in  $\mcL^{\infty}_r$ to examine possible scenarios of singularities, in particular, when they are large.
	On the other hand,  if a solution has  locally finite energy, i.e.,  $\nabla \Bu\in L^2_{loc}$,
	then $\Bu \in L^{2n/(n-2)}_{\mathrm{loc}}\subset L^n_{loc}$  when $n\leq 4.$
	Therefore, one can get the regularity of a solution with  finite local energy only if $n\leq 4$.
	Hence one may ask whether there is a singular  weak solution $\Bu$ with $\nabla \Bu \in L^2_{loc} (\mathbb{R}^n)$ in the case $n\geq 5$.
	One should note that the solutions in $\mcL^{\infty}_r$  have locally finite energy when $n\geq 5$. Hence
	a particularly interesting case is $n=5$ as five is the smallest dimension when such a solution has locally finite energy.
	Therefore, as pointed out in \cite{Sverak11}, the singular behavior of order $1/|x|$ is most severe in $n=5$ among the high dimensions $n\geq 5$.
	
	As mentioned before, 
it follows from  \cite{Tsai98b} and \cite[Corollary in Section 2]{Sverak11}) that
	if $\Bu\in C^\infty(\mathbb{R}^n\setminus \{0\})$  is a steady self-similar  solution to \eqref{NS} in  $\mathbb{R}^n\setminus\{0\}$, then $\Bu \equiv 0$ ($n\geq 4)$ or $\Bu$ is a Landau solution ($n=3$). Note that the Landau solutions solve ${\eqref{NS}}_1$ with the right-hand side replaced by a singular force which is a multiple of Dirac measure. In other words, the Landau solutions  do not solve  the Navier-Stokes system \eqref{NS} across the origin in $\mathbb{R}^3$.
 Thus one can conclude that all the $(-1)$-homogeneous very weak solutions of the Navier–Stokes equations in $\mathbb{R}^n$ ($n\geq 3$) are trivial.
 This  eliminates the simplest conceivable singularity which belongs to $\mcL^{\infty}_r$.
	Then \u Sver\'ak in \cite{Sverak11} raised an interesting question:
	\emph{whether one can have a nontrivial very weak solution $\Bu$ in $\mathbb{R}^n$  which is smooth away from the
		origin  satisfying $|\Bu(x)|\leq C/|x|$ in $\mathbb{R}^n$.}
	Theorem \ref{main} resolves this problem for all higher dimensional cases $n\geq 4$.


	As an  application of Theorem \ref{main}, we can show that for solutions of steady Navier-Stokes system \eqref{NS} in dimensions $n\geq 4$,   the singularity at the origin  which behaves as  $\frac{1}{|x|}$ is removable.  This removes the smallness assumption appeared in \cite[Theorem 4]{KimKozono06} and \cite{KPR20JMFM}.
	\begin{theorem}\label{cor:regularity}
		Let $\Bu$ be a smooth solution to \eqref{NS}  in $B_1\setminus\{0\}$, $n\geq 4$. Assume $\Bu$ satisfies
		\begin{equation}\label{cor1_cond1}
			|\Bu(x) |\leq \frac{C}{|x|}\quad \text{for all}\, \,x\in B_1\setminus\{0\},
		\end{equation}
		then $\Bu$ is regular in $B_1$.    
	\end{theorem}
We have a few remarks about Theorem \ref{cor:regularity}.
		\begin{remark}
		Note that in Theorem \ref{cor:regularity}, we  assume neither integrability of $\Bu$ or its derivative nor the smallness of the constant $C$ in \eqref{cor1_cond1}.
		Moreover, the domain $B_1$ in Theorem \ref{cor:regularity} can be replaced by any fixed ball $B_R$, $R>0$.
	\end{remark}

		\begin{remark}
A similar problem   for the case $n=3$ is open, it is necessary to require the solutions to solve  the Navier-Stokes system \eqref{NS} across the origin to obtain its regularity on $B_1\setminus\{0\}$ in view of the Landau solutions.
   This problem is also closely related to the rigidity of the Landau solutions, which has been already explained. Moreover, the leading order singular behavior at the origin of a solution defined on $B_1 \setminus\{0\}\subset\mathbb{R}^3$ is also given by Landau solutions if the solution
	satisfies \eqref{cor1_cond1} with a small constant $C$ (cf. \cite{MiuraTsai12}).
	\end{remark}

	
	\begin{remark}
		
		Let $\boldsymbol{f}\in L^\infty (\Omega)$, where  $\Omega$ is a bounded domain in $\mathbb{R}^n$, $n\geq 5 $ with $\partial \Omega \in C^{0,1}$. It was proved in  \cite[Theorem 1]{FrehseRuzicka98} that every weak solution $\Bu \in W^{1,2}_0 (\Omega)$ to \eqref{NS_2} 
   must be regular,  provided
     that $\Bu$ satisfies  the Dirichlet boundary condition $\Bu =0$ on $\partial \Omega$ and the integrability condition $\Bu \in L^q (\Omega)$ with $q\geq 4$ and $q>n/2$.
		The singular behavior $|\Bu(x)|\leq C/|x|$ satisfies the $ L^q$ condition near $x=0$, but the Dirichlet boundary condition is not met in Theorem \ref{cor:regularity}. Hence one may not directly  \cite[Theorem 1]{FrehseRuzicka98} to remove the singularity at $x=0$ in Theorem \ref{cor:regularity}.
	\end{remark}


Besides  helping to understand the regularity of solutions, the rigidity of steady solutions also plays an important role in characterizing the asymptotic behaviors of the solutions. As another application of Theorem \ref{main}, we can prove the following theorem, which asserts that the solutions in exterior domains bounded by  $\frac{C}{|x|}$ must decay at an optimal rate $\frac{C}{|x|^{n-2}}$.
	\begin{theorem}\label{thm:asy}
		Let $\Bu$ be a smooth solution to the Navier-Stokes equations \eqref{NS} in $\mathbb{R}^n \setminus B_1 $, $n\geq 4$. If $\Bu$  satisfies
		\begin{equation}\label{corollary1-condition}
			|\Bu(x)|\leq \frac{C}{|x|} \quad \text{for all}\,\, x\in \mathbb{R}^n\setminus B_1 
		\end{equation}
		for some constant $C>0$, then
		there exists a constant $\tilde{C}>0$ such that
		\be \label{3-2-1}
		|\Bu(x)| \leq \frac{\tilde{C}}{|x|^{n-2}}\quad \textrm{for all  } x\in \mathbb{R}^n\setminus B_1.
		\ee
	\end{theorem}
	There are several remarks for Theorem \ref{thm:asy}.
	\begin{remark}
	Note that in  Theorem \ref{thm:asy}, we assume neither square integrability of $\nabla \Bu$ (i.e., a finite energy condition) nor the smallness of the constant $C$ in \eqref{corollary1-condition}.
	Moreover, $B_1$ appeared in Theorem \ref{thm:asy} can be replaced by any fixed ball $B_R$, $R>0$.
\end{remark}
	\begin{remark}
	The decay rate \eqref{3-2-1} is optimal in the sense that the fundamental solutions of the Stokes system decay in the same way. On the other hand, using the Navier-Stokes system with compactly supported external force $f$,  one can easily construct solutions whose leading terms are $G*f$, which decays at the rate $\frac{1}{|x|^{n-2}}$. Here, $G$ is the fundamental solution of the Stokes system. This also shows that the decay rate \eqref{3-2-1} is optimal in general.
\end{remark}

\begin{remark}
		We should mention that the asymptotic behaviors of  solutions for the steady Navier-Stokes system in the two or three-dimensional case should be very different. For example, Landau solutions are proved to be the leading order term in the asymptotic expansion near infinity of a small solution in the three-dimensional exterior domain in \cite{KorolevSverak11}.
	The asymptotic behaviors of the solutions at infinity in two-dimensional exterior domains are much harder and seem largely open, since one needs to find more self-similar solutions other than what has been found so far, to parameterize the general asymptotic behavior  as pointed out in \cite{GuillodWittwer15}.
\end{remark}



	The major technique for the proof of Theorem \ref{main}
	is to establish weighted energy estimates for \eqref{NS} and the equation of  the total head pressure  $H=\frac{|\Bu|^2}{2}+p$ (see Equation \eqref{heq2} below for instance) over annulus.
	Then we can establish  the  integrabilities of certain crucial quantities, which imply a faster decay   at infinity and a milder blow-up rate at the origin of the solution compared to what was initially assumed. This eventually leads to the triviality of $\Bu$.

	The proof for the problem in the four-dimensional case seems a little bit simpler. We can prove that $\nabla \Bu \in L^2 (\mathbb{R}^4)$, which implies faster decay and milder blow-up of $\Bu$ and leads to the triviality of $\Bu$.
	In the case dimension $n\geq 5$,
	one cannot directly get some integrability of the velocity $\Bu$ and its gradient due to an additional term involving the total head pressure $H$.
	To handle this term, we first prove the non-positivity of  $H$ by using the governing equation for the total head pressure (cf. \eqref{heq2}). Then we choose suitable multipliers and use weighted energy estimates
	to prove some weighted integrability of $H$ in  $\mathbb{R}^n$, which eventually leads to $\Bu \equiv 0$. The proof of Theorems \ref{cor:regularity} and  \ref{thm:asy} are based on the  blow-up arguments and Theorem \ref{main}.

	The rest of the  paper is organized as follows. Section \ref{Prelim} is devoted to  preliminary  estimates for the velocity and the pressure. Next, the rigidity of steady solutions, i.e., Theorem \ref{main}, is proved in Section \ref{4d}.  In  Section \ref{sec:appl}, we apply Theorem \ref{main} to prove  the  singularity of steady solutions  at the origin bounded by  $\frac{C}{|x|}$ are removable (Theorem \ref{cor:regularity}) and asymptotic behaviors of solutions in the exterior domains (Theorem \ref{thm:asy}).
	 An elementary yet  useful lemma, which asserts the limiting behaviors of a non-negative function at the origin and far-field when it is integrable on the positive half line, is  presented in Appendix \ref{appen1}.
	\section{Preliminaries} \label{Prelim}
	This section is devoted to some preliminary estimates  regarding the behaviors  of the velocity and the pressure and their derivatives, which play an important role in the  proof of  Theorem \ref{main}.



	\begin{lemma} \label{lemma21}
		Let $\Bu\in C^\infty (\mathbb{R}^n\setminus \{ 0\}), n\geq 4,$  be a solution to \eqref{NS} in $\mathbb{R}^n \setminus \{0\}$,  which satisfies \eqref{eq:sacinv}.
		Then it holds that for some constants $C_l$
		\begin{align} \label{eq13}
			|\nabla ^l \Bu(x)| \leq \frac{C_l}{|x|^{1+l}} , \quad
			|\nabla^{l}p| \leq \frac{C_l}{|x|^{2+l}}
			\quad \text{in }\mathbb{R}^n \setminus \{0\}, \quad l =0,1,2,\ldots
		\end{align}
		where the pressure  $p$ is defined up to a constant.
	\end{lemma}
	\begin{proof}
		One can just follow the proof of   \cite[Lemma X.9.2]{Galdi11} word for word to get ${\eqref{eq13}}$ for $l\geq1$ (but not the estimate of $p$ when $l=0$).  When going through the proof there, note that we do not need  $|x|$ to be large because the corresponding external force $\boldsymbol{f}=0$ and $|\Bu|\leq C/|x|$ in the entire space in our situation. Hence for  $ l\geq 1$,  the estimate {\eqref{eq13}} holds for all $x\in \mathbb{R}^n\setminus\{0\}$ rather than just large $x$.
		In fact, the argument was originally developed   in \cite{SverakTsai00}.
		
		Now it remains to prove the estimate for the pressure $p$. It is noted that
		\begin{equation}\label{estgradp}
			|\nabla p(x) | \leq \frac{C_1}{|x|^{3}}, \ \ \ \text{for all }x\in \mathbb{R}^n \setminus \{0\}.
		\end{equation}		 
		 Let
		\begin{equation}\nonumber
			\bar{p}_k = \frac{1}{|\partial B_{2^k} |}
			\int_{\partial B_{2^k} } p (x) \, d\sigma,\ \ \ \ \ \text{for}\,\, k\in \mathbb{Z}.
		\end{equation}
	The straightforward computations yield
		\begin{equation}\nonumber
			\begin{aligned}
			| \bar{p}_{k+1} - \bar{p}_k | = &	\left| \frac{1}{|\partial B_{2^{k+1}} | } \int_{\partial B_{2^{k+1}} } p(x) d\sigma  - \frac{1}{|\partial B_{2^k} | } \int_{\partial B_{2^k} } p(x)  \, d\sigma   \right|\\ 
			= & \left| \frac{1}{|\partial B_{2^k} | } \int_{\partial B_{2^k} } (p(2x) - p(x)  )\, d\sigma   \right| \\
			\leq & C  2^{-2k} ,
			\end{aligned}
		\end{equation}
	where the estimate \eqref{estgradp} has been used to get the last inequality. This, in particular, implies that $\{\bar{p}_k\}$ has a limit as $k\to +\infty$. 
		Subtracting $p$ by some constant, we can assume that
		\begin{equation}\nonumber
			\bar{p}_{k} \rightarrow 0\,\, \,\, \text{as}\,\, k\to +\infty \, \, \, \, \mbox{and} \, \, \, \, |\bar{p}_k| \leq C  2^{-2k}\,\, \text{for all } k\in \mathbb{Z}.
		\end{equation}
		There is a point $x^k$ on the sphere $\partial B_{2^k} $, such that $p(x^k)= \bar{p}_k$ and hence $ |p(x^k)|\leq C2^{-2k}$.
		For any $ x \in B_{2^{k+1}} \setminus B_{2^k}  $, there exists a piecewise smooth curve $\ell\subset \overline{B_{2^{k+1}}\setminus B_{2^k}}$ which connects $x^k$ and $x$ and has length less than $ 2^{k+3} \pi$.
		Therefore, using \eqref{estgradp}, one has
		\be \label{pressure1}
		|p (x)|=\left|\int_{\ell}\frac{\partial p}{\partial \tau}\, ds +p(x^k)\right|\leq \frac{C}{2^{2k}} \leq \frac{C}{|x|^2}.
		\ee
	This completes the proof of Lemma \ref{lemma21}.		
	\end{proof}

One can also prove an analogous statement in $\mathbb{R}^n_+\setminus \{0\}$ for the velocity field and the pressure.
\begin{lemma} \label{lemma22}
	Let $\Bu\in C^\infty (\Omega)$, $\Omega= \mathbb{R}^n_+\setminus \{0\}, n\geq 4,$  be a solution to \eqref{NS} in $\Omega$,  which satisfies \eqref{eq4} and \eqref{eq4_1} in $\Omega$.
	Then it holds that for some constant $C$
	\begin{align} \label{eq13_3}
		|\nabla \Bu(x)| + |p(x)| \leq \frac{C}{|x|^{2}} \quad \ \ \text{in }\Omega,
	\end{align}
	where the pressure $p$ is defined up to a constant.
\end{lemma}
\begin{proof}
	Fix $x\in \Omega$, and consider the following two cases.
	
	{\it Case 1. $x_n>\frac{1}{15n}|x|$}. Let $R=\frac{|x|}{45n}$. Clearly, one has
	$Ry+x \in \mathbb{R}^n_+$ for all $y\in B_2$.
	Let
	$$\Bu_R(y) = R\Bu (Ry+x),	\quad 	 p_R(y)= R^2 p(Ry+x).$$
	As $\Bu$ satisfies \eqref{eq:sacinv},  one has
	\begin{align*}		|\Bu_R(y)| \leq \frac{CR}{|Ry+x|}
		\leq C\,\,\text{ for }y\in B_2,
	\end{align*}
	where	the constant $C$  depends neither on $R$ nor on $y$.
	Note that $\Bu_R$ and $p_R$ satisfy
	\begin{align*}
		-\Delta \Bu_R (y)
		+ \Bu_R\cdot \nabla \Bu_R (y)
		+ \nabla p _R (y)=0,
		\quad\ \ \  \div \Bu_R (y) =0	\,\ \ \mbox{in}\  B_2.
	\end{align*}
	By the regularity theory of the Stokes equations (cf. \cite[Theorem IV.4.4 and  Remark IV.4.2]{Galdi11}), one can get
	\begin{align} \label{eq_17_}
		\|\Bu_R\|_{W^{1,q}(B_1)}\ &\leq  C \left(  \|\Bu_R\|_{L^{2q} (B_2)}^2 + \|\Bu_R\|_{L^q (B_2)} \right)\leq C, \quad 1<q<\infty.
	\end{align}
	Here, the constant $C$ does not depend on $R$.
	Next, one can use the regularity theory again (cf. \cite[Theorem IV.4.1 and Remark IV.4.1]{Galdi11}) and \eqref{eq_17_} to obtain
	\begin{align*}
		\|\Bu_R\|_{W^{2,q}(B_{1/2})} \leq
		C \left( \|\Bu_R \cdot \nabla \Bu_R\|_{L^{q}(B_1)}
		+\|\Bu_R\|_{W^{1,q}(B_1)} \right)\leq C,
	\end{align*}
	and 
	\begin{equation}
		\|\Bu_R\|_{W^{3,q}(B_{1/4})} \leq
		C \left( \|\Bu_R \cdot \nabla \Bu_R\|_{W^{1, q}(B_{1/2} )}
		+\|\Bu_R\|_{W^{2,q}(B_{1/2} )} \right)\leq C.
	\end{equation}
	Hence by Sobolev embedding, it holds that
	$$	|\nabla \Bu_R(0)| + |\nabla^2 \Bu_R(0) |  \leq C\|\Bu_R\|_{W^{3,q}(B_{1/4})} \leq C,\  \ \ \ q>n.$$
	This implies
	\begin{align*}
		|\nabla \Bu(x)|\leq CR^{-2}\leq C|x|^{-2},\ \ \ \ \ |\nabla^2 \Bu(x)|\leq CR^{-3} \leq C |x|^{-3}. 
	\end{align*}
	
	{\it Case 2. $x_n\leq \frac{1}{15n}|x|$}. Let $\bar{x}=(x_1, \cdots, x_{n-1},0)$ and $\bar{R}=\frac{1}{5{n}}|\bar{x}|$. Clearly,  one has
	\begin{equation}\label{eqxxbar}
		\frac{4\sqrt{14}}{15}|x|\leq |\bar{x}|\leq |x|.
	\end{equation}
	Furthermore,   there is a $\hat{y}\in V_{\frac{2}{3}}$ such that $x=\bar{R} \hat{y}+\bar{x}$, where
	\[
	V_s :=\{(y_1, \cdots, y_{n-1}, y_n)|\,  y_n\in (0,s),\,  y_i\in (-s, s) \,\,\text{for}\,\, i=1,\cdots, n-1\}.
	\]
	Let
	$$\bar\Bu_{\bar{R}}(y) = \bar{R}\Bu (\bar{R}y+\bar{x}),	\quad 	 \bar{p}_{\bar{R}}(y)= {\bar{R}}^2 p(\bar{R}y+\bar{x}).$$
	As $\Bu$ satisfies \eqref{eq:sacinv}, one has
	\begin{align*}		|\bar\Bu_{\bar{R}}(y)| \leq \frac{C\bar{R}}{|\bar{R}y+\bar{x}|}
		\leq C\,\, \  \text{ for }y\in V_2(0),
	\end{align*}
	where	the constant $C$  depend neither on $\bar{R}$ nor on $y$.
	Note that $\bar\Bu_{\bar{R}}$ and $\bar{p}_{\bar{R}}$ satisfy
	\begin{align*}
		-\Delta \bar\Bu_{\bar{R}} (y)
		+ \bar{\Bu}_{\bar{R}}\cdot \nabla \bar\Bu_{\bar{R}} (y)
		+ \nabla \bar{p}_{\bar{R}} (y)=0,
		\quad \div \bar\Bu_{\bar{R}} (y) =0	\,\, \ \text{for}\,\, y\in V_2
	\end{align*}
	supplemented with the boundary conditions
	\begin{equation}\label{eq16}
		\bar{\Bu}_{\bar{R}}(y)=0\quad  \text{for} \,\, y\in  \partial V_2\cap \partial\mathbb{R}_+^n.
	\end{equation}
	By the regularity theory of the Stokes equations (cf. \cite[Theorem IV.5.3 and  Remark IV.5.2]{Galdi11}), one can get
	\begin{align}\label{eq_19}
		\|\bar\Bu_{\bar{R}}\|_{W^{1,q}(V_1)}\ &\leq  C \left(  \|\bar\Bu_{\bar{R}}\|_{L^{2q} (V_2)}^2 + \|\bar\Bu_{\bar{R}}\|_{L^q (V_2)} \right)\leq C, \quad 1<q<\infty,
	\end{align}
	where we have used the boundary condition \eqref{eq16}. Here, the constant $C$ does not depend on $R$.
	Next, one can use the regularity theory for Stokes equations up to the boundary again (cf. \cite[Theorem IV.5.1, Remark IV.5.1]{Galdi11}) and \eqref{eq_19} to obtain
	\begin{align*}
		\|\bar\Bu_{\bar{R}}\|_{W^{2,q}(V_{3/4})} \leq
		C \left( \|\bar\Bu_{\bar{R}} \cdot \nabla \bar\Bu_{\bar{R}}\|_{L^{q}(V_1)}
		+\|\bar\Bu_{\bar{R}}\|_{W^{1,q}(V_1)} \right)\leq C
	\end{align*}
 and 
 \begin{align*}
     \|\bar\Bu_{\bar{R}}\|_{W^{3,q}(V_{2/3})} \leq
		C \left( \|\bar\Bu_{\bar{R}} \cdot \nabla \bar\Bu_{\bar{R}}\|_{W^{1,q}(V_{3/4} )}
		+\|\bar\Bu_{\bar{R}}\|_{W^{2,q}(V_{3/4} )} \right)\leq C.
 \end{align*}
	Hence by Sobolev embedding, it also holds that
	$$	|\nabla \bar\Bu_{\bar{R}}(\hat{y})| + |\nabla^2 \bar{\Bu}_{\bar{R}} (\hat{y})|  \leq C\|\bar\Bu_{\bar{R}}\|_{W^{3,q}(V_{2/3})} \leq C\ \ \ \ \text{for}\,\, q>n.$$
	This, together with \eqref{eqxxbar}, implies
	\begin{align*}
		|\nabla \Bu(x)|\leq C|x|^{-2},\ \ \ \ \ \ |\nabla^2 \Bu(x)|\leq C|x|^{-3}.
	\end{align*}

The estimate for the pressure can be proved in the similar way as that in Lemma \ref{lemma21} when we replace the spheres $\partial B_{2^k}$ by the half spheres $(\partial B_{2^k})_+=\{x: |x|<2^k, x_n\geq 0\}$.
	Hence the proof of the lemma is completed.
\end{proof}


\section{Rigidity of steady solutions} \label{4d}
This section is devoted to the proof of Theorem \ref{main}. We first derive a weighted energy estimate for solutions in $\mathbb{R}^n\setminus \{0\}$ with $n\geq4$. Then   the case $n=4$ is investigated in Subsection \ref{subsec_4d}.  Next, the  case $n\geq 5$ is analyzed in Subsection \ref{subsec_5d}. The  rigidity of  solutions in $\mathbb{R}^n_{+}$ with Navier slip boundary condition is proved in Subsection \ref{sec_Navier}. At last, the proof of rigidity of solutions in the four-dimensional half space with Dirichlet boundary condition is presented in Subsection \ref{sec_half}, which is quite similar to the analysis for the solutions in $\mathbb{R}^4\setminus \{0\}$.

Let $\Bu$ be a solution to \eqref{NS} in $\mathbb{R}^n\setminus \{0\}, n\geq 4$, which satisfies \eqref{eq:sacinv}.
Then by Lemma \ref{lemma21}, the total head pressure given by $H=\frac{|\Bu|^2}{2}+p$ satisfies
\begin{align} \label{eq18}
	|\nabla^k H(x)| \leq \frac{C_k}{|x|^{2+k}}\quad \text{in }\mathbb{R}^n \setminus \{0\}, \quad k=0,1,2,\ldots
\end{align}

Note that
\begin{align}
	\Delta \Bu \cdot \Bu= \Bu \cdot \nabla \Bu \cdot \Bu  +\nabla p \cdot \Bu
	=
	\Bu \cdot \nabla \left( \frac{|\Bu|^2}{2} \right)
	+ \Bu \cdot \nabla p
	=\Bu \cdot \nabla H.
\end{align}
Multiplying the above equation  by $r^{4-n}$ (here, $r=|x|$), and integrating it over the annulus $B_{R_2}\setminus B_{R_1}$ centered at the origin  yield
\begin{equation}
	\begin{aligned} \label{eq13_2}
		0=&  -\int_{B_{R_2}\setminus B_{R_1}}
		\Delta \Bu \cdot \Bu \thinspace r^{4-n} dx
		+ \int_{B_{R_2}\setminus B_{R_1}}
		(\Bu \cdot \nabla) H \thinspace r^{4-n} dx  \\
		=&\int_{B_{R_2}\setminus B_{R_1}}
		(|\nabla \Bu|^2 r^{4-n}+(n-4)|\Bu|^2\thinspace r^{2-n})dx\\
		&~
		+
		\left(
		\int_{ \partial B_r}
		-\partial_r \Bu \cdot \Bu \thinspace  r^{4-n}
		-\frac{n-4}{2} |\Bu|^2 r^{3-n}
		d\sigma\right)
		\Bigg|_{r=R_1}^{r=R_2} \\
		&~+ (n-4)\int_{B_{R_2}\setminus B_{R_1}}
		H \thinspace u^r \thinspace r^{3-n} dx
		+
		\left(
		\int_{ \partial B_r} u^r H r^{4-n}  d\sigma \right)\Bigg|_{r=R_1}^{r=R_2},
	\end{aligned}
\end{equation}
where $u^r=\Bu\cdot \Be_r=\Bu\cdot\frac{x}{|x|}$.
Alternatively, we can write
\begin{align} \label{eq17_2}
    \int_{B_{R_2}\setminus B_{R_1}}
	(|\nabla \Bu|^2 r^4 + (n-4)|\Bu|^2 r^2 + (n-4) H u^r r^3 )\thinspace r^{-n} dx
	=
	h_1(R_2)-h_1(R_1),
\end{align}
where
\begin{align} \label{eq18_2}
	h_1(R)= \int_{\partial B_R}
	\left(\partial_r \Bu \cdot \Bu \thinspace R^{4-n}
	+ \frac{n-4}{2} |\Bu|^2 R^{3-n} - u^r \thinspace H \thinspace R^{4-n} \right) d\sigma.
\end{align}
When $n=4$, Equation \eqref{eq17_2} has a much simpler form   whereas for $n\geq 5$, it looks a little more complicated. Hence we carry on the analysis for these two cases separately.

\subsection{\textit{Case $n=4$}} \label{subsec_4d}
In this case, Equations \eqref{eq17_2} and \eqref{eq18_2} reduce to
\begin{align} \label{eq19_3}
	\int_{B_{R_2}\setminus B_{R_1}} |\nabla \Bu|^2 dx
	=h_1(R_2)-h_1(R_1),
	~ \textrm{where}~
	h_1(R)=
	\int_{\partial B_R}
	(\partial_r \Bu \cdot \Bu
	- u^r \thinspace H ) \thinspace  d\sigma.
\end{align}
Hence $h_1(R)$ is increasing with respect to $R$. Moreover, it follows from Equation \eqref{eq18} and Lemma \ref{lemma21}  that  $h_1 \in L^\infty (0,\infty)$. Therefore, \eqref{eq19_3} implies $\nabla \Bu\in L^2 (\mathbb{R}^4)$.
Hence, as $\lim_{|x|\to\infty}\Bu(x)=0$, one can obtain, by the Sobolev inequality (cf. Theorem II.6.1 of \cite{Galdi11}),
\begin{align*}
	\int_0^\infty
	\left(
	\int_{\partial B_r} |\Bu|^4 d\sigma
	\right) dr = \int_{\mathbb{R}^4} |\Bu|^4 dx <\infty.
\end{align*}
Thus it follows from  Lemma \ref{lem:fasterdecay} that
\begin{align}\label{eq28}
	\liminf_{R\to 0} \left( R \int_{\partial B_R} |\Bu|^4 d\sigma \right)=\liminf_{R\to \infty} \left( R \int_{\partial B_R} |\Bu|^4 d\sigma \right) =0.
\end{align}
Now an  application of the H\"older inequality together with Lemma \ref{lemma21} and  \eqref{eq18}  yields that
\begin{equation}\label{eq:Holder1}
	\begin{aligned}
		|h_1(R)|&\leq      \left( \int_{\partial B_R} |\Bu|^4 d\sigma \right)^{\frac{1}{4}}\left( \int_{\partial B_R} |\partial_r\Bu|^{\frac{4}{3}}d\sigma \right)^{\frac{3}{4}}
		+ \left( \int_{\partial B_R} |u^r|^4 d\sigma \right)^{\frac{1}{4}}\left( \int_{\partial B_R} |H|^{\frac{4}{3}}d\sigma \right)^{\frac{3}{4}}\ \\
		&\leq C R^{\frac{1}{4}} \left( \int_{\partial B_R} |\Bu|^4 d\sigma \right)^{\frac{1}{4}}.
	\end{aligned}
\end{equation}
Using \eqref{eq28} and \eqref{eq:Holder1} and noting that  $h_1$ is increasing, we  have
\begin{equation}
	\lim_{R\to 0} h_1(R)=\liminf _{R\to 0}h_1(R)=0 \textrm{ and } \lim_{R\to \infty} h_1(R)=\liminf _{R\to \infty}h_1(R)=0.
\end{equation}
Taking the limits as $R_2\to\infty$ and $R_1\to 0$ in \eqref{eq19_3}, one can obtain $\|\nabla \Bu\|_{L^2 (\mathbb{R}^4)}=0 $ and thus $\Bu=constant$, which further implies $\Bu=0$ due to $|\Bu(x)|\leq\frac{C}{|x|}$. This finishes the proof for the first part of Theorem \ref{main} in the case $n=4$.

\subsection{\textit{Case $n\geq 5$}} \label{subsec_5d}
We prove the first part of Theorem \ref{main} in the case $n\geq 5$.
In this case, we cannot directly apply the same idea to \eqref{eq17_2} as for the case $n=4$ mainly due to the presence of the term involving $H\thinspace u^r$ on the left-hand side of \eqref{eq17_2}. To handle this term, we first prove some weighted estimates of $H$ via  energy estimates.

We begin with the following equation of the total head pressure
\begin{align} \label{heq2}
	-\Delta H+ \Bu\cdot \nabla H =-2 |\boldsymbol{\omega}|^2  \quad \text{in }\mathbb{R}^n\setminus \{0\},
\end{align}
where $\boldsymbol{\omega}$ is the anti-symmetric part of $\nabla \Bu$.
Multiplying   Equation \eqref{heq2} by $H_+^\alpha$, $\alpha=(n-4)/2,$ and integrating
over the annulus $B_{R_2} \setminus B_{R_1}$ centered at the origin yield
\begin{equation} \label{eqalpha}	
	\begin{aligned}
		-2 \int  _{B_{R_2} \setminus B_{R_1}}  |\boldsymbol{\omega}|^2 H_+^\alpha \thickspace dx
		&=
		-\int _{B_{R_2} \setminus B_{R_1}} \Delta H H_+^\alpha \thickspace dx + \int  _{B_{R_2} \setminus B_{R_1}}  \Bu \cdot \nabla H H^{\alpha}_+ \thickspace dx \\
		&= \alpha \int _{B_{R_2}\setminus B_{R_1}}  |\nabla  H_+|^2 H^{\alpha-1}_+ \thickspace dx
		+h_2(R_2)-h_2(R_1)
	\end{aligned}
\end{equation}
where
\begin{align*}
	h_2(R)= - \int_{\partial B_R}
	\left(
	\partial_r H_+ H_+^\alpha - \frac{1}{\alpha +1} u^r H_+^{\alpha+1}
	\right)
	d\sigma.
\end{align*}
Therefore, one can get
\begin{align} \label{eq53}
	h_2(R_2)-h_2(R_1)
	=
	-\alpha \int _{B_{R_2} \setminus B_{R_1}}  |\nabla  H_+|^2 H^{\alpha-1}_+ \thickspace dx
	-2 \int  _{B_{R_2} \setminus B_{R_1}}  |\boldsymbol{\omega}|^2 H_+^{\alpha} \thickspace dx\leq 0.
\end{align}
This implies that $h_2(R)$ is a decreasing function on $(0,\infty)$.
On the other hand, $h_2\in L^\infty (0,\infty)$ due to \eqref{eq18}.
Hence, both  the limits,
$\lim_{R\to 0}h_2(R)$ and $\lim_{R \to \infty} h_2(R)$ exist and are finite.
Then taking the limits as $R_2\to\infty$ and  $R_1\to0$ in \eqref{eq53} simultaneously, we obtain
\begin{equation}
	\begin{aligned} \label{eq54}
		\alpha \int _{ \mathbb{R}^n }  |\nabla  H_+|^2 H^{\alpha-1}_+ \thickspace dx
		+2 \int  _{ \mathbb{R}^n }  |\boldsymbol{\omega}|^2 H_+^{\alpha} \thickspace dx
		=\lim_{R\to0 } h_2(R)
		-\lim_{R\to\infty} h_2(R)
		<\infty.
	\end{aligned}
\end{equation}
Note $\alpha=(n-4)/2>0$.
Hence $\nabla (H_+^{\frac{\alpha+1}{2}}) \in L^2 (\mathbb{R}^n)$. 
Therefore, as $\lim_{|x|\to\infty}H_+^{\frac{\alpha+1}{2}}(x)=0$, one can obtain $H_+^{\frac{\alpha+1}{2}} \in L^{\frac{2n}{n-2} }(\mathbb{R}^n)$ by the Sobolev inequality (cf. Theorem II.6.1 of \cite{Galdi11}), which means
\begin{align*}
	\int_{0}^\infty
	\int_{ \partial B_r} H_+^{\frac{n}{2}} d\sigma
	dr
	= \int_{\mathbb{R}^n} H_+^{\frac{n}{2}} dx<\infty.
\end{align*}
Thus it follows from Lemma \ref{lem:fasterdecay} that
\begin{align}\label{eq:decayofH}
	\liminf_{R\to 0}\left(
	R \int_{\partial B_R }
	H_+^{\frac{n}{2}} d\sigma\right) =\liminf_{R\to \infty}
	\left(R \int_{\partial B_R }
	H_+^{\frac{n}{2}} d\sigma \right)
	=0.
\end{align}
The H\"older inequality   together with Lemma \ref{lemma21} and \eqref{eq18} yields
\begin{equation}
	\begin{aligned}
		|h_2(R)| &\leq \left(\int_{\partial_{B_R}}(H_+^\alpha)^{\frac{n}{n-4}}\right)^{\frac{n-4}{n} }
		\left(\int_{\partial_{B_R}}|\partial_r H_+|^{\frac{n}{4}}\right)^{\frac{4}{n}}
		+ \left(\int_{\partial_{B_R}}(H_+^\alpha)^{\frac{n}{n-4}}\right)^{\frac{n-4}{n} }
		\left(\int_{\partial_{B_R}}|u_r H_+|^{\frac{n}{4}}\right)^{\frac{4}{n}}\\
		& \leq CR^{1-\frac{4}{n}}\left(\int_{\partial B_{R}}H_+^{\frac{n}{2}} d\sigma \right)^{1-\frac{4}{n}}.
	\end{aligned}
\end{equation}
Hence, using the limits in \eqref{eq:decayofH} and noting that $h_2(R)$ is a decreasing function, it follows that
\begin{equation*}
	\lim_{R\to0 }h_2(R)=\lim_{R\to\infty}h_2(R)=0.
\end{equation*}
Therefore, from \eqref{eq54}, it holds that $H_+\equiv 0$, that is
$H\leq 0 \text{ in }\mathbb{R}^n \setminus \{0\}$.

Now  multiplying the momentum equation  of \eqref{NS} by $\nabla {r^{4-n}}$, integrating it over the annulus $B_{R_2}\setminus B_{R_1}$ for $ R_2>R_1$ give
\begin{align}\label{eq:momeq}
	\int_{B_{R_2}\setminus B_{R_1}}
	(-\Delta \Bu + \Bu \cdot \nabla \Bu + \nabla p)
	\cdot
	\nabla {r^{4-n}} dx=0.
\end{align}
For the first term above, we claim that
\begin{align}
	\label{eq+28}
	-\int_{B_{R_2}\setminus B_{R_1}}
	\Delta \Bu \cdot \nabla  r^{4-n}  dx
	=0.
\end{align}
To see this, we apply the integration by parts along with $\div (\Delta \Bu )=0$ to obtain
\begin{align} \label{eq:divfree}
	-\int_{B_{R_2}\setminus B_{R_1}}
	\Delta \Bu \cdot \ \nabla {r^{4-n}}  dx = -\int_{B_{R_2}\setminus B_{R_1}}
	\div(\Delta \Bu   {r^{4-n}} ) dx
	=-\left(
	{r^{4-n}}\int _{\partial B_r}
	\Delta \Bu \cdot \Be_r \thinspace d\sigma
	\right) \Bigg|_{r=R_1}^{r=R_2}.
\end{align}
By using the divergence theorem with $\div (\Delta \Bu)=0$ and Lemma \ref{lemma21}, one can get that for any fixed $r>0$
\begin{align*}
	\left|
	\int_{\partial B_r}
	\Delta \Bu \cdot \Be_r
	\thinspace d\sigma
	\right|
	=
	\left|
	\int _{\partial B_\varepsilon} \Delta \Bu \cdot \Be_r d\sigma
	\right|
	\leq C \varepsilon^{n-4} \to 0 \quad \text{as }\varepsilon \to 0,
\end{align*}
where $n\geq 5$ has been used.
This, together with \eqref{eq:divfree} gives  \eqref{eq+28}.

For the second and third terms on the left-hand side of \eqref{eq:momeq}, integration by parts yields
\begin{align} \label{eq27}
	\begin{aligned}
		&\int_{B_{R_2}\setminus B_{R_1}}
		\left(
		\Bu \cdot \nabla \Bu
		+ \nabla p
		\right)
		\cdot  \nabla {r^{4-n}}
		dx\\
		=& (4-n)\left\{
		\int_{B_{R_2}\setminus B_{R_1}}
		\left(
		(n-2)|u^r|^2 - |\Bu|^2
		-2p
		\right) r^{2-n } \thinspace dx
		+
		\left(
		\int_{ \partial B_r }
		(|u^r|^2 +p) r^{3-n}
		d\sigma\right) \Bigg|_{r=R_1}^{r=R_2} \right\}.
	\end{aligned}
\end{align}
Combining \eqref{eq+28} and \eqref{eq27}, and due to $H\leq 0$ in $\mathbb{R}^n\setminus\{0\}$, one has
\begin{align} \label{eq31}
	h_3(R_2)- h_3(R_1)=
	\int_{B_{R_2}\setminus B_{R_1}}
	\frac{1}{r^{n-2}}
	\left(
	(n-2)|u^r|^2
	-2H
	\right) dx
	\geq 0,
\end{align}
where
\begin{align}
	\label{eq27_2}
	h_3(R)
	= -\int_{\partial B_R}
	\left(
	|u^r|^2
	+p
	\right) R^{3-n} d\sigma.
\end{align}
Again $h_3$ is an increasing function, and $h_3\in L^\infty (0, \infty)$ thanks to Lemma \ref{lemma21}. Hence the limits,
$\lim_{R\to\infty} h_3(R)$ and $\lim_{R\to 0} h_3(R)$ exist and are finite. Sending  $R_2\to\infty$ and $R_1 \to 0$ in \eqref{eq31}, one can obtain
\begin{align} \label{eq35_2}
	\int_{\mathbb{R}^n}
	\left(
	(n-2)\frac{|u^r|^2}{r^{n-2}}
	-\frac{2 H}{r^{n-2}}
	\right) dx
	=
	\lim_{R\to\infty} h_3(R)
	-\lim_{R\to 0}h_3(R)
	<\infty.
\end{align}
Since  $H\leq 0$ in $\mathbb{R}^n \setminus \{0\}$, one can get
\begin{align} \label{eq66}
	\int_{\mathbb{R}^{n}}
	\frac{|H|}{r^{n-2}} dx
	<\infty.
\end{align}


Now it follows from \eqref{eq66} and \eqref{eq:sacinv} that one has $H u^r r^{3-n} \in L^1 (\mathbb{R}^n)$. With this integrability, we can argue similarly to  the case $n=4$. For $h_1$ defined in \eqref{eq18_2},
Lemma \ref{lemma21} and \eqref{eq18} shows that $h_1\in L^\infty (0,\infty)$. Hence both  $\liminf_{R\to 0} h_1(R)$ and $\liminf_{R\to \infty} h_1(R)$  are finite. First  taking the lower limit as $R_2\to\infty$ on both sides of \eqref{eq17_2} then taking the upper limit on both sides of \eqref{eq17_2} as $R_1\to 0$, and using $H u^r r^{3-n} \in L^1 (\mathbb{R}^n)$, one obtains
\begin{equation}\label{equgradu}
	\int_{\mathbb{R}^n}|\Bu|^2 r^{2-n}+|\nabla \Bu|^2r^{4-n}<\infty.
\end{equation}
This, together with \eqref{eq66}, yields
\begin{equation}
	\int_{\mathbb{R}^n} |p|r^{2-n}dx \leq \int_{\mathbb{R}^n}\left|H-\frac{|\Bu|^2}{2}\right| r^{2-n}dx <\infty.
\end{equation}
Hence one has
\begin{equation*}
	\int_{0}^{\infty}\frac{|h_3(R)|}{R}dR\leq \int_{0}^{\infty}\int_{\partial{B_R}}\left||u^r|^2+p\right|R^{2-n}d\sigma dR = \int_{\mathbb{R}^n}\left||u^r|^2+p\right|r^{2-n} dx <+\infty.
\end{equation*}
It then follows from Lemma \ref{lem:fasterdecay} that
\begin{equation*}
	\liminf_{R\to 0} |h_3(R)|
	=\liminf_{R\to \infty} |h_3(R)|=0.
\end{equation*}
Note that $h_3$ is an increasing function. Therefore, it holds that
\begin{equation*}
	\lim_{R\to 0} h_3(R)
	=\lim_{R\to \infty} h_3(R)=0.
\end{equation*}
Finally sending $R_1\to 0$ and $R_2\to \infty$
in \eqref{eq31} and using $H\leq 0$, we deduce that $u^r= H\equiv0$.

It follows from Lemma \ref{lem:fasterdecay} and \eqref{equgradu} that  we have
\begin{align}\label{eq:decayofu}
	\liminf_{R\to 0}
	\left(
	R^{3-n} \int_{\partial B_R} |\Bu|^2 d\sigma
	\right)=\liminf_{R\to \infty}
	\left(
	R^{3-n} \int_{\partial B_R} |\Bu|^2 d\sigma
	\right)=0.
\end{align}
This,  together with the property $H\equiv 0$,  \eqref{eq18_2}, the H\"older inequality, and Lemma \ref{lemma21}  gives     $$\liminf_{R\to 0} |h_1|(R)=\liminf_{R\to \infty} |h_1|(R)=0.$$
Furthermore, we immediately know from \eqref{eq17_2} and the property $H \equiv 0$ that $h_1$ is an increasing function. Hence one has
$$\lim_{R\to 0} |h_1|(R)=\lim_{R\to \infty} |h_1|(R)=0.$$
Therefore, taking the limit as $R_2\to \infty$ and $R_1\to 0$ on both sides of  \eqref{eq17_2}, one can directly get $\Bu\equiv 0$.
This finishes the proof  for the first part of Theorem \ref{main} in the case $n\geq 5$.
\subsection{Rigidity of solutions in $\mathbb{R}^n_{+}$ with Navier slip boundary condition} \label{sec_Navier}
This subsection is devoted to the proof for the second part of Theorem \ref{main}. The main idea is extending the soltuion $\Bu$ to the whole space $\mathbb{R}^n$. Note that the boundary condition \eqref{eq4_0} can be written as 
\begin{equation}\nonumber
u_n = 0, \ \ \ \ \partial_n u_i = 0, \ i=1, \cdots, n-1. 
\end{equation}
Let 
\begin{equation*}
\tilde{\Bu}(x^\prime, x_n)= \left\{ \begin{array}{l} (u_1 (x^\prime, x_n),  \cdots, u_{n-1}(x^\prime, x_n), \, u_n(x^\prime, x_n)), \ \ \ \ \  (x^\prime, x_n) \in \mathbb{R}^n_{+}, \\
(u_1 (x^\prime, -x_n),  \cdots, u_{n-1}(x^\prime, -x_n), \, -u_n(x^\prime, -x_n)), \ \ \ (x^\prime, x_n) \in \mathbb{R}^n_{-}, \end{array}
\right. 
\end{equation*}
and
\begin{equation*}
\tilde{p}(x^\prime, x_n)= \left\{ \begin{array}{l} p(x^\prime, x_n), \ \ \ \ \  (x^\prime, x_n) \in \mathbb{R}^n_{+}, \\
p(x^\prime, - x_n), \ \ \ (x^\prime, x_n) \in \mathbb{R}^n_{-}. \end{array}
\right. 
\end{equation*}
One can check that $(\tilde{\Bu}, \tilde{p})$ is the solution to the Navier-Stokes equations in $\mathbb{R}^n \setminus \{0\}$, and satisfies 
 \begin{equation}\label{estutilde}
     |\tilde{\Bu}(x)| \leq \frac{C}{|x|}, \ \ \ x\in \mathbb{R}^n \setminus \{0\}.
 \end{equation}
 With the aid of \eqref{estutilde}, we can immediately show that  $(\tilde{\Bu}, \tilde{p})$ is a smooth solution of \eqref{NS} in $\mathbb{R}^n \setminus \{0\}$.
 According to the first part of Theorem \ref{main}, $\tilde{\Bu} \equiv 0$ and the proof  for the second part of Theorem \ref{main} is completed.

 \subsection{Rigidity of solutions in four-dimensional half space} \label{sec_half}
This subsection is devoted to the proof for the third part of Theorem \ref{main}.

Multiplying the momentum equation in \eqref{NS} by $\Bu$ and integrating over $B_{R_2,+}\setminus B_{R_1,+},R_2>R_1>0$ where $B_{R,+}=\{x:|x|<R, x_n>0\}$ denotes the half ball, one obtains
\begin{align} \label{eq49}
	\int_{B_{R_2,+}\setminus B_{R_1,+}}
	|\nabla \Bu|^2 dx
	= h(R_2)-h(R_1),
\end{align}
where we have used the boundary condition and
\begin{align*}
	h(R)
	=-
	\int_{(\partial B_{R})_+}
	\left(
	\frac{\partial \Bu}{\partial r} \cdot \Bu - \frac{1}{2}|\Bu|^2 u^r - p u^r \right) \thinspace d\sigma,
	\textrm{ with }  (\partial B_R)_+ = \partial B_R \cap \mathbb{R}^4_+.
\end{align*}
Hence $h(R)$ is an increasing function.

It follows from \eqref{eq4} and Lemma \ref{lemma22} that one has  $h\in L^\infty (0,\infty)$.
Consequently, both of the limits $\lim_{R\to 0}h(R)$ and $\lim_{R\to\infty}h(R)$ exist and are finite. Letting $R_1\to 0, R_2\to\infty$ in \eqref{eq49} yields that
$\nabla \Bu \in L^2 (\mathbb{R}^4_+)$.

Moreover, it follows from  \eqref{eq4_1} and the Sobolev embedding theorem that
$\Bu\in L^4 (\mathbb{R}^4_+)$.
Thus the following limits hold:
\begin{align} \label{eq37_2}
	\liminf_{R\to 0}
	\left( R
	\int_{ (\partial B_R)_+ } |\Bu|^4 d\sigma \right)=\liminf_{R\to \infty}
	\left( R
	\int_{ (\partial B_R)_+ } |\Bu|^4 d\sigma \right)=0.
\end{align}
Then one can get $$\lim_{R\to 0}h(R)=\lim_{R\to\infty}h(R)=0$$ by using \eqref{eq4} and \eqref{eq37_2} with the help of the H\"older inequality similar to \eqref{eq:Holder1}.
Hence, sending $R_1\to 0$ and $ R_2\to \infty$ in  \eqref{eq49}, one can get
$ \int_{\mathbb{R}^4_+}        |\nabla \Bu|^2 dx=0$,
which
implies $\Bu \equiv 0$. This finishes the proof for the third part of Theorem \ref{main}.

\begin{remark}
    Note that we don't have detailed information for the pressure on the boundary $\partial \mathbb{R}^n_+$, so it does not seem easy to adapt the methods developed in Subsection \ref{subsec_5d} directly to prove the rigidity of solutions in half space with dimension $n\geq 5$.
\end{remark}

\section{Applications of Rigidity for steady solutions}\label{sec:appl}
In this section, we use Theorem \ref{main} to prove Theorems \ref{cor:regularity} and \ref{thm:asy}.

\subsection{Removable singularity for steady solutions}
In this subsection, as an application of Theorem \ref{main}, we prove Theorem  \ref{cor:regularity}.
\begin{proof}[Proof of Theorem \ref{cor:regularity}]
	We only need to show the regularity of solutions at the origin $0$. It follows from \cite{KimKozono06} that if $\Bu$ is not regular at the origin, then  there exist a constant $\epsilon>0$ and a sequence of points  $\{x^k\}$, such that $x^k \to 0$ and
	\begin{equation}\label{eq:sequence}
		|\Bu(x^k)|\geq \frac{\epsilon}{|x^k|}.
	\end{equation}
	Let $\lambda_k=|x^k|$ and
	\begin{equation*}
		\Bu_k(x)= \lambda_k \Bu(\lambda_k x),~~\ x\in B_{\frac{1}{\lambda_k}}(0).
	\end{equation*}
	According to \eqref{corollary1-condition}, we have
	\begin{equation}\label{eq47}
		|\Bu_k(x)|\leq \frac{C}{|x|}.
	\end{equation}
	Then by the regularity theory for the Navier-Stokes equations,
	for every fixed compact subdomain $\Omega$ in $\mathbb{R}^n \setminus \{0\}$,
	$\Bu_k, \nabla \Bu_k, \nabla^2 \Bu_k,  \nabla^3 \Bu_k$ are uniformly bounded in $\Omega$, for $k$ large enough.
	It then follows that after passing to  a subsequence, $\Bu_k\to \Bar{\Bu}$  in $C^2_{loc}(\mathbb{R}^n \setminus \{0\})$ for some vector field $\Bar{\Bu}$.
	Hence $\bar{\Bu}$ is a solution to \eqref{NS} on $\mathbb{R}^n\setminus\{0\}$.
	
	Moreover, it follows from \eqref{eq47} that $|\bar{\Bu}(x)|\leq \frac{C}{|x|}$.
	Now, we can apply Theorem \ref{main} to conclude that $\bar{\Bu}\equiv 0$. On the other hand, the estimate \eqref{eq:sequence} implies that for any $k$,
	\begin{equation*}
		\left|\Bu_k\left(\frac{x^k}{\lambda_k}\right)\right|\geq \epsilon.
	\end{equation*}
	As $\{\lambda_k^{-1} x^k\}$ is a sequence of points on the unit sphere,    after passing to a subsequence and using the uniform convergence of $\Bu_k$, we have for some $x^*\in \partial B_1$ that $|\bar{\Bu}(x^*)|\geq \epsilon$. This leads to a contradiction and thus $\Bu$ must be regular at the origin.
\end{proof}

\subsection{Asymptotic behavior for steady solutions in exterior domains}\label{sec:asy}
This subsection is devoted to the proof of Theorem \ref{thm:asy}.
To prove Theorem \ref{thm:asy}, it suffices to prove the following two lemmas.

\begin{lemma}\label{corollary-2-1}
	Assume that $\Bu$ is a smooth solution to the Navier-Stokes equations \eqref{NS} in $\mathbb{R}^n \setminus B_1 $ and satisfies
	\be \label{3-1-1}
	|\Bu(x)|\leq \frac{C}{|x|}, \ \  x\in \mathbb{R}^n\setminus B_1.
	\ee
	Then  we have
	\be \label{3-2-1.5}
	|\Bu(x)| =o\left(\frac{1}{|x|}\right)\  \textrm{ as } x\to \infty.
	\ee
\end{lemma}
\begin{proof}
	We prove \eqref{3-2-1.5} by a contraction argument similar to Theorem \ref{cor:regularity}.
	If  \eqref{3-2-1.5} does not hold, there exist an $\epsilon>0$ and a sequence of points $\{x^k\}$ with $|x^k|\geq k$ such that $\dis |\Bu(x^k)| \geq \frac{\epsilon}{|x^k|}$.
	Let
	\be \nonumber
	\lambda_k = |x^k|,\ \ \ \Bu_k(x) = \lambda_k \Bu( \lambda_k x) \ \ \ \text{for}\,\,  x\in \mathbb{R}^n\setminus B_{\lambda_k^{-1}}.
	\ee
	Then $\Bu_k$ is a solution to the Navier-Stokes equations in $\mathbb{R}^n\setminus B_{\lambda_k^{-1}} $ and satisfies $\dis |\Bu_k(x)| \leq \frac{C}{|x|}$.
	Moreover, it follows from the regularity theorem for the Navier-Stokes system that for every fixed compact subdomain $\Omega$ in $\mathbb{R}^n \setminus \{0\}$,
	$\Bu_k, \nabla \Bu_k, \nabla^2 \Bu_k,  \nabla^3 \Bu_k$ are uniformly bounded in $\Omega$ for $k$ large enough. Hence, there is a subsequence of $\{\Bu_k\}$ (still labelled by $\{\Bu_k\}$) which converges to some $\bar{\Bu}$ in  $C^2_{loc}(\mathbb{R}^n \setminus \{0\})$. This implies that $\bar{\Bu}$ is a solution to the Navier-Stokes equations in $\mathbb{R}^n\setminus \{0\}$ and satisfies $\dis |\bar{\Bu}(x)| \leq \frac{C}{|x|}$. According to Theorem \ref{main}, one has $\bar{\Bu} \equiv 0$.
	
	However, as assumed,
	\be\nonumber
	|\Bu_k (\lambda_k^{-1} x^k) | = |\lambda_k \Bu (x^k)| \geq \frac{\epsilon \lambda_k}{|x^k|} =\epsilon.
	\ee
	As $\{\lambda_k^{-1} x^k\}$ is a sequence of points on the unit sphere,  there is a subsequence of $\{\lambda_k^{-1} x^k\}$ (still labeled by $\{\lambda_k^{-1} x^k\}$) which converges to a point $x^* $ on the unit sphere so that
	\be \nonumber
	|\bar{\Bu}(x^*) | = \lim_{k \rightarrow \infty}| \Bu_k(\lambda_k^{-1} x^k )| \geq \epsilon.
	\ee
	This leads to a contradiction and thus the lemma is proved.
\end{proof}

\begin{lemma}\label{epsilon-R}
	Assume that $\Bu$ is a smooth solution to the Navier-Stokes equations \eqref{NS} in $\mathbb{R}^n \setminus B_R$. There exists an absolute constant $\epsilon<1$, such that if $\Bu$ satisfies
	\begin{equation}\label{2-1}
		|\Bu(x)| \leq \frac{\epsilon}{|x|}\ \ \ \text{for}\,\, x\in \mathbb{R}^n\setminus B_R,
	\end{equation}
	then it holds that
	\begin{equation} \label{2-2}
		|\Bu(x)| \leq \frac{CR^{n-3} }{|x|^{n-2}} \ \ \ \text{for}\,\,x\in \mathbb{R}^n\setminus B_R.
	\end{equation}
	where the constant $C=C(\epsilon)$ is independent of $R$ and $\Bu$.
\end{lemma}
\begin{remark}
	When $n=4$,  Lemma \ref{epsilon-R} is essentially the same as \cite[Lemma 3.2]{JiaSverak17} where  the solutions are required  to have  finite energy additionally. However, when $n=4$, under the  assumption $|\Bu(x)| =O\left( \frac{1}{|x|}\right)$, one can follow the proof in Section \ref{subsec_4d} to prove $\nabla \Bu \in L^2(\mathbb{R}^4\setminus B_1)$.
\end{remark}
\begin{proof}[Proof of Lemma \ref{epsilon-R}] For any $\Bu$ satisfying Navier-Stokes system \eqref{NS} and \eqref{2-1} in $\mathbb{R}^n\setminus B_R$,  $\Bu_R (x) = R \Bu(R x)$ is a solution of  the Navier-Stokes equations on $\mathbb{R}^n \setminus B_1$ and satisfies \eqref{2-1} with $R=1$. If the lemma is true for $R=1$, then one has
	\begin{equation}\label{2-3}
		|\Bu(x)| = R^{-1} |\Bu_R(R^{-1} x ) | \leq C R^{-1} |R^{-1} x|^{2-n} = CR^{n-3} |x|^{2-n}.
	\end{equation}
	Hence we need only to prove the lemma for $R=1$.
	The rest of the proof is inspired by the method developed in \cite{KorolevSverak11} and is divided into four steps. We first consider the solution with zero flux, i.e., \begin{equation}\label{eq:zeroflux}
		\int_{\partial B_\rho} \Bu \cdot \Bn \, d\sigma =0 \ \ \  \mbox{for every}\ \rho>1.
	\end{equation}
	Finally, in Step 4, we prove the lemma without this constraint.
	
	{\it Step 1. Extension of $\Bu$ to $\tilde{\Bu}$ on the whole space $\mathbb{R}^n$.}
	Let $\eta$ be a smooth cut-off function satisfying
	\be \label{4-1}
	\eta(r)= 1\, \,  \text{for}\,\, r\geq \frac52 \ \ \ \mbox{and}\  \ \ \eta(r) = 0\ \  \text{for}\,\, r\leq 2.
	\ee
	Let $\tilde{\Bu} =\eta\Bu+\Bv$, where $\Bv$ has compact support in $B_3$ and solves the divergence equation
	\be \label{4-2}
	{\rm div}\,\Bv = - \Bu \cdot \nabla \eta.
	\ee
	Indeed,  such $\Bv$ can be constructed via the Bogovskii formula as long as the zero flux condition \eqref{eq:zeroflux} holds, one may refer to \cite[Chapter III.3]{Galdi11} for details. Moreover, such $\Bv$ is smooth and compactly supported in $B_3$  as $\Bu \cdot \nabla \eta$ is smooth and of compact support in $B_3$ (cf. \cite[Theorem III.3.3]{Galdi11}).
 Therefore, $\tilde{\Bu}$ satisfies
	\begin{equation}
		-\Delta\tilde{\Bu}+\tilde{\Bu}\cdot \nabla \tilde{\Bu}+\nabla {\tilde{p}} ={\Bf},~ \ \ \rm div~\tilde{\Bu}=0.
	\end{equation}
	for some smooth ${\Bf}$ compactly supported in  $B_3$ and $\tilde{p} = \eta p$. In fact,
	\begin{equation}\nonumber
		\Bf = -2 \nabla \eta\cdot \nabla \Bu - \Delta\eta \Bu
		- \Delta \Bv + (\eta \Bu \cdot \nabla)((\eta -1) \Bu) +
		(\eta \Bu \cdot \nabla)\Bv + (\Bv \cdot \nabla) (\eta \Bu)+ (\Bv\cdot \nabla )\Bv + p \nabla \eta.
	\end{equation}
	
	One can just follow the proof of   \cite[Lemma X.9.2]{Galdi11} to prove that
	\begin{equation}\label{4-1-1}
		|\nabla^l \Bu(x)| \leq \frac{C(l)\epsilon}{|x|^{l+1}}, \ \ \  x\in \mathbb{R}^n\setminus {B_2}, \ \ l\in \mathbb{N}.
	\end{equation}
	It follows from  the properties of the  Bogovskii operator (cf. \cite[Theorem III.3.3]{Galdi11}), the Sobolev embedding and \eqref{4-1-1} that  $\Bv$ satisfies
	\begin{equation}\label{60-0}
		\|\nabla^l \Bv\|_{L^{\infty}}\leq C(l) \epsilon.
	\end{equation}
		Following the same lines as that in the proof of  Lemma \ref{lemma21}, one can get the estimate for the pressure term,
		\begin{equation}\label{4-1-2}
			|p(x)| \leq \frac{C\epsilon}{|x|^2},\ \  \ \ x\in \mathbb{R}^n \setminus B_2.
		\end{equation}
		In summary, \eqref{4-1-1}-\eqref{4-1-2} give that
		\begin{equation}\label{eq:decayfortildeu}
			|\tilde{\Bu}(x)|\leq \frac{C \epsilon}{1+|x|}.
		\end{equation}
		and
		\begin{equation}\label{4-1-3}
			\|\Bf \|_{L^\infty}\leq C\epsilon.
		\end{equation}
		
		{\it Step 2. Construction of a solution $\Bw$ which decays as $|x|^{2-n}$.} The major goal of this step is to  show that for any $\Bf$ compactly supported in  $B_3$ with $\|\Bf\|_{L^{\infty}} \leq\epsilon_0$ for some absolute constant $\epsilon_0$, one can construct a solution $\Bw$ of the system
		\begin{equation*}
			-\Delta \Bw+ \Bw\cdot \nabla \Bw +\nabla {q} ={\Bf}, ~ \ \ {\rm div~} \Bw=0,
		\end{equation*}
		which satisfies \begin{equation}\label{eq:decay for w}
			|\Bw(x)|\leq \frac{C\epsilon_0}{(1+|x|)^{n-2}}.
		\end{equation}
		We construct such $\Bw$ by a fixed point argument.
		Let $G =\left( G_{ij} \right)$ be the Green tensor of the linear Stokes operator. One has (cf. \cite{Galdi11})
		\[
		| G(x)|\leq \frac{C}{|x|^{n-2}}\quad \text{and}\quad |\nabla G(x)|\leq \frac{C}{|x|^{n-1}}.
		\]
		
		For $\alpha>0$,	 let
		\[
		X_\alpha=\left\{\Bw: \ \Bw\in C^1(\mathbb{R}^n; \mathbb{R}^n), \ {\rm div}\,\Bw =0, \ \sup_{x\in \mathbb{R}^n} (1+|x|)^{\alpha}|\Bw(x)| <\infty\right\}.
		\]
		and \[
		\|\Bw\|_{X_{\alpha}} =\sup_{x\in \mathbb{R}^n}(1+|x|)^{\alpha}|\Bw(x)|.
		\]
		We will seek a solution in $X_{n-2}$  by a fixed point argument
		for the equivalent equation
		\begin{equation}\label{fixeqw}
			\Bw= -B(\Bw, \Bw)+G*{\Bf},
		\end{equation}
		where for $\mfu$, $\mfv\in X_{n-2}$,
		\[
		B(\mfu,\mfv)=G*\textrm{div}(\mfu\otimes\mfv).
		\]
		Note that
		\begin{equation} \nonumber \begin{aligned}
				|G*{\Bf}(x)| & \leq C \int_{\mathbb{R}^n} \frac{1}{|x-y|^{n-2}} |\Bf(y)| \, dy 
				 \leq C \epsilon_0 \int_{B_3} \frac{1}{|x-y|^{n-2}} \, dy   \leq \frac{C\epsilon_0}{(1+ |x|)^{n-2}}.
			\end{aligned}
		\end{equation}
		Under the smallness assumption on ${\Bf}$, the fixed point will be achieved if we can prove
		\begin{equation}\label{eq:bilinear}
			\|B(\mfu, \mfv)\|_{X_{n-2}}\leq C\|\mfu\|_{X_{n-2}}\|\mfv\|_{X_{n-2}}.
		\end{equation}
		Indeed, integration by parts yields
		\begin{equation}\label{eq:bouofbili}
			\begin{aligned}
				|B(\mfu, \mfv)(x)|&\leq C\int_{\mathbb{R}^n}\frac{1}{|x-y|^{n-1}}|\mfu(y)\otimes\mfv(y)|\, dy\\
				&\leq C\|\mfu\|_{X_{n-2}}\|\mfv\|_{X_{n-2}}\int_{\mathbb{R}^n}\frac{1}{|x-y|^{n-1}}\frac{1}{(1+|y|)^{2(n-2)}}\, dy.\\
			\end{aligned}
		\end{equation}
		Define
		\[
		I(x)= \int_{\mathbb{R}^n}\frac{1}{|x-y|^{n-1}}\frac{1}{(1+|y|)^{2(n-2)}}\,dy.
		\]
		Clearly, $I(x)$ is bounded for $|x|\leq 1$. For $|x|> 1$, let $x=te$ with $|e|=1$ and $t=|x|> 1$, then let   $y = tz$ to obtain
		\begin{equation*}
			\begin{aligned}
				I(x)= & t^{-(2n-5)}\int_{\mathbb{R}^n}\frac{1}{|e-z|^{n-1}}\frac{1}{(t^{-1}+|z|)^{2(n-2)}}\,dz\\
				\leq & t^{-(2n-5)}\int_{|z|\leq \frac{1}{2}}\frac{1}{|e-z|^{n-1}}\frac{1}{(t^{-1}+|z|)^{2(n-2)}}\,dz\\
		&+ t^{-(2n-5)}\int_{|z|\geq \frac{1}{2}}\frac{1}{|e-z|^{n-1}}\frac{1}{(t^{-1}+|z|)^{2(n-2)}}\,dz \\
    \leq & Ct^{-(2n-5)} \left( \int_0^{\frac12} \frac{1}{(t^{-1} + r)^{2(n-2)}} \cdot r^{n-1}\, dr 
     + \int_{|z|\geq \frac{1}{2}}\frac{1}{|e-z|^{n-1}}\frac{1}{|z|^{2(n-2)}}\, dz \right)\\ 
    =&  Ct^{-(2n-5)}\left( \int_0^{\frac12} \frac{1}{(t^{-1} + r)^{2(n-2)}} \cdot r^{n-1}\, dr 
     + \int_{|\tilde{z}|\geq \frac{1}{2}}\frac{1}{|e_1-\tilde{z}|^{n-1}}\frac{1}{|\tilde{z}|^{2(n-2)}}\, d\tilde{z} \right)  \\
				\leq & C(n)t^{-(2n-5)}t^{n-3} + C(n)t^{-(2n-5)} \\
    = & C(n)t^{2-n}+C(n)t^{-(2n-5)},
			\end{aligned}
		\end{equation*}
  where $e_1=(1,0,0,\cdots,0)$ and we have used the rotation invariance of the domain $\{|z|\geq 1/2\}$ and the associated integral.
		Combining the boundedness of $I(x)$ when $|x|\leq 1$  and the above equation for $|x|> 1$, we have
		\begin{equation*}
			(1+|x|)^{n-2}I(x)\leq C ~\ \ \textrm{for all } x\in \mathbb{R}^n.
		\end{equation*}
		This, together with \eqref{eq:bouofbili}, shows that \eqref{eq:bilinear} holds.
		The classical fixed point theorem  (cf. \cite[Lemma 2]{KorolevSverak11}) gives the desired $\Bw$ satisfying \eqref{eq:decay for w} and \eqref{fixeqw}.

		{\it Step 3. Equality of $\tilde{\Bu}$ and $\Bw$.} Let $\Bf$ in  {Step 2} be constructed via {Step 1} and let $\epsilon$ small enough so that $\|{\Bf}\|_{L^{\infty}}\leq \epsilon_0$ is satisfied.
		Let $\tilde{\Bw} = \tilde{\Bu}-\Bw$, we will show that $\tilde{\Bw}=0$ if $\epsilon$ is small enough. The straightforward computations show that $\tilde{\Bw}$ satisfies
		\begin{equation*}
			-\Delta \tilde{\Bw}+  \tilde{\Bw}\cdot\nabla {\Bw} + \tilde{\Bu}\cdot \nabla \tilde{\Bw}+\nabla(\tilde{p}-q)=0.
		\end{equation*}
		It follows from \eqref{eq:decayfortildeu} and \eqref{eq:decay for w} that one has
		\begin{equation*}
			|\tilde{\Bw}(x)|\leq\frac{C\epsilon}{1+|x|}.
		\end{equation*}
		Note \begin{equation*}
			\tilde{\Bw}(x)= G*\textrm{div}(\tilde{\Bw}\otimes\Bw+ \tilde{\Bu}\otimes\tilde{\Bw}).
		\end{equation*}
		We have
		\begin{equation*}
			\begin{aligned}
				| G*\textrm{div}(\tilde{\Bw}\otimes\Bw+ \tilde{\Bu}\otimes\tilde{\Bw})(x)|&\leq
				C\epsilon\|\tilde{\Bw}\|_{X_1}\int_{ \mathbb{R}^n }\frac{1}{|x-y|^{n-1}}\frac{1}{(1+|y|)^2}dy\\
			\end{aligned}
		\end{equation*}
		It is bounded by $C\epsilon\|\tilde{\Bw}\|_{X_1}$ for $|x|\leq 1$. For $|x|> 1$,  we let $x=te$ with $|e|=1$ and $t=|x|>1$, and make the  change of variables $y = tz$ to obtain (with $n\geq 4$)
		\begin{equation*}
			\begin{aligned}
				| G*\textrm{div}(\tilde{\Bw}\otimes\Bw+ \tilde{\Bu}\otimes\tilde{\Bw})(x)|\leq & C\epsilon\|\tilde{\Bw}\|_{X_1}\frac{1}{|x|}\int_{ \mathbb{R}^n }\frac{1}{|e-z|^{n-1}}\frac{1}{(t^{-1}+|z|)^2}dz\\
    = & C\epsilon\|\tilde{\Bw}\|_{X_1}\frac{1}{|x|}\int_{ \mathbb{R}^n }\frac{1}{|e_1-\tilde{z}|^{n-1}}\frac{1}{(t^{-1}+|\tilde{z}|)^2}d\tilde{z}\\
				\leq&  C \epsilon\frac{\|\tilde{\Bw}\|_{X_1}}{|x|}.
			\end{aligned}
		\end{equation*}
		This leads to
		\begin{equation*}	\|\tilde{\Bw}\|_{X_1}\leq C \epsilon\|\tilde{\Bw}\|_{X_1}
		\end{equation*}
		and implies that $\tilde{\Bw}\equiv0$ if $\epsilon$ small enough so that $C \epsilon<1$. This finishes the proof of the lemma if $\Bu$ satisfies the zero flux condition \eqref{eq:zeroflux}.
		
		{\it Step 4. Removal of the zero flux constraint.}
		If the flux is non-zero,
		define
		\[
		\boldsymbol{a}=c\frac{x}{|x|^{n}}\quad \text{where}\quad c=\frac{1}{|\partial B_1|}\int_{\partial B_1} \Bu \cdot \Bn\, d\sigma.
		\]
		Note $\textrm{div}\,\boldsymbol{a}=0$ and $\Delta\,\boldsymbol{a}=0$.
		Moreover, one can directly check
		$\boldsymbol{a}$ satisfies the Navier-Stokes system \eqref{NS} with the associated
		pressure $p_{\boldsymbol{a}}= -\frac{1}{2}|\boldsymbol{a}|^2$ in $\mathbb{R}^n\setminus\{0\}$.
		Let $\Bu=\Bu_1 +\boldsymbol{a}$ so that $\Bu_1$ satisfies the  zero flux  condition.
		We  can now extend $\Bu_1$ to a divergence-free vector field $\tilde{\Bu}_1$ on $\mathbb{R}^n$ similar to {Step 1}.
		It can be verified that $\tilde{\Bu}_1$ satisfies
		\begin{equation}
			-\Delta \tilde{\Bu}_1 + \tilde{\Bu}_1 \cdot \nabla \tilde{\Bu}_1 + \nabla (\eta p - \eta p_{\boldsymbol{a}} ) =
			\Bf \ \  \ \mbox{in}\ \mathbb{R}^n,
		\end{equation}
		where
		\begin{equation}\nonumber
			\begin{aligned}
				\Bf  = &-2 \nabla \eta\cdot \nabla \Bu_1 - \Delta \eta \Bu_1
				- \Delta \Bv + (\eta \Bu_1 \cdot \nabla)((\eta -1) \Bu_1) +
				(\eta \Bu_1 \cdot \nabla)\Bv + (\Bv \cdot \nabla) (\eta \Bu_1)\\ &\ \ + (\Bv\cdot \nabla )\Bv
				+( p-p_{\boldsymbol{a}}) \nabla \eta -(\eta \Bu\cdot \nabla) \boldsymbol{a} - (\eta \boldsymbol{a} \cdot \nabla) \Bu.
			\end{aligned}
		\end{equation}
		Although $\Bf$ is not compactly supported in $B_3$, it also holds that
		\begin{equation}\nonumber
			|\Bf(x)| \leq \frac{ C \epsilon}{(1+ |x|)^{n+1}}
		\end{equation}
		and
		\begin{equation}\nonumber
			|G*\Bf(x)| \leq \frac{C\epsilon}{(1+|x|)^{n-2}}.
		\end{equation}
		The proof follows the same lines as above, so we omit the details and finish the proof of the lemma.
\end{proof}

\appendix
\section{An elementary lemma} \label{appen1}

This appendix is devoted to
an elementary yet quite useful lemma, which is about the decay rate at infinity and the blow-up rate at the origin of a non-negative integrable function compared with $\frac{1}{r}$, and has been used several times in the proof of Theorem \ref{main}.

\begin{lemma}\label{lem:fasterdecay}
	Let $f(x)$ be a non-negative function on $(0,+\infty)$, if
	\begin{equation*}
		\int_{0}^{+\infty}f(x)\, dx<+\infty,
	\end{equation*}
	then \begin{equation*}
		\liminf_{x\to +\infty}xf(x)=\liminf_{x\to 0}xf(x)=0.
	\end{equation*}
\end{lemma}
\begin{proof}
	We prove the lemma by a contradiction argument. If
	\begin{equation*}
		\liminf_{x\to +\infty}xf(x)=c>0,
	\end{equation*}
	then there exists $N>0$ such that
	\begin{equation*}
		xf(x)\geq \frac{c}{2}, \textrm{ for any } x>N.
	\end{equation*}
	This implies that
	\begin{equation*}
		\int_{0}^{+\infty}f(x)\, dx\geq \int_{N}^{+\infty} \frac{c}{2x} \, dx=+\infty,
	\end{equation*}
	and leads to a contraction. Hence one has \begin{equation*}
		\liminf_{x\to +\infty}xf(x)=0.
	\end{equation*}
	The case $\liminf_{x\to 0}xf(x)=0$ can be proved similarly. Therefore, the proof of the lemma is completed.
\end{proof}

{\bf Acknowledgement.}
The research of Gui is partially supported by NSF grant DMS-1901914. The research of Wang is partially supported by NSFC grants 12171349 and 12271389. The research of  Xie is partially supported by  NSFC grants 11971307 and 1221101620,  Natural Science Foundation of Shanghai 21ZR1433300,  and Program of Shanghai Academic Research Leader 22XD1421400.

\end{document}